\documentclass[10pt,a4paper]{article} 

\usepackage[english]{babel}
\usepackage{abstract}
\usepackage{amsmath}
\usepackage{amsfonts}
\usepackage{amssymb}
\usepackage{suffix}
\usepackage{mathtools}
\usepackage{framed}
\usepackage[svgnames]{xcolor}
\usepackage[amsmath, amsthm, thmmarks, framed]{ntheorem}
\usepackage{comment}
\usepackage{tikz}
\usetikzlibrary{calc, decorations.markings}
\usepackage[textsize=footnotesize]{todonotes}
\usepackage{hyperref}

\definecolor{lightgray}{gray}{1}

\newtheorem{theorem}{Theorem}[section]
\newtheorem{corollary}[theorem]{Corollary}
\newtheorem{proposition}[theorem]{Proposition}
\newtheorem{lemma}[theorem]{Lemma}
\theoremstyle{remark}
\newtheorem{remark}[theorem]{Remark}

\renewenvironment{proof}[1][\textit{Proof.}]{\begin{trivlist}
\item[\hskip \labelsep {\bfseries #1}]}{\end{trivlist}}

\makeatletter
\def\blfootnote{\xdef\@thefnmark{}\@footnotetext}
\makeatother

\DeclarePairedDelimiterX\MeijerM[3]{\lparen}{\rparen}{\begin{matrix}#1 \\ #2\end{matrix}\delimsize\vert\,#3}

\newcommand\MeijerG[7]{ G^{\,#1,#2}_{#3,#4}\MeijerM*{#5}{#6}{#7}}

\newcommand{\ds}{\displaystyle}
\DeclareMathOperator{\Tr}{Tr}
\DeclareMathOperator{\Pf}{Pf}

\DeclareMathOperator{\Real}{Re}
\DeclareMathOperator{\Imag}{Im}
\renewcommand{\Re}{\Real}
\renewcommand{\Im}{\Imag}

\DeclareMathOperator{\diag}{diag}

\newcommand{\U}{\mathcal U}

\newcommand{\eins}{\leavevmode\hbox{\small1\kern-3.8pt\normalsize1}}

\numberwithin{equation}{section}

\title{Singular value statistics of matrix products with truncated unitary matrices}
\author{Mario Kieburg, Arno B. J. Kuijlaars and Dries Stivigny}
\date{\today}

\begin{document}

\maketitle
\blfootnote{MK: Fakult\"at  f\"ur Physik, Universit\"at Bielefeld, Postfach 100131, D-33501 Bielefeld,
Germany, e-mail: mkieburg@physik.uni-bielefeld.de \\ 
AK and DS: KU Leuven, Department of Mathematics, Celestijnenlaan 200B box 2400, 3001 Leuven, Belgium, 
e-mails: arno.kuijlaars@wis.kuleuven.be  and dries.stivigny@wis.kuleuven.be}


\begin{abstract} 
We prove that the squared singular values of a fixed matrix multiplied with a truncation of
a Haar distributed unitary matrix are distributed by a polynomial ensemble.
This result is applied to a multiplication of a truncated unitary matrix with a random matrix.
We show that the structure of polynomial ensembles and of certain Pfaffian ensembles is preserved. Furthermore we derive the joint singular value density of a product of truncated unitary matrices and its corresponding correlation kernel which can be written as a double contour integral. This leads to hard edge scaling limits that also include new finite rank perturbations of the Meijer G-kernels found for products of complex Ginibre random matrices.
\end{abstract}


\section{Introduction}

The study of  random matrices benefits greatly from explicit
formulas of joint eigenvalue densities that are known for large classes
of random matrix ensembles. Quite a lot of these densities have the structure of a
determinantal or Pfaffian point process. Such structures are incredibly helpful to extract the spectral statistics in the limit of large matrix dimension of such ensembles. Various techniques as free probability \cite{gb2}, orthogonal polynomials \cite{Bor,Methabook,Dei,Kon}, and supersymmetry \cite{gb1} have been developed to derive these limits.

In an important recent development it was found that explicit formulas
also exist for the eigenvalue and singular value statistics of products of random matrices.
This was first established for the eigenvalues \cite{AkBu,Ips,IpKi,For14} and singular values \cite{AkIpKi, AkKiWe}
of products of Ginibre matrices. Shortly after this development, results were also derived for the eigenvalues
of products of truncated unitary matrices \cite{AdReReSa,AkBuKiNa,IpKi}. A common feature is that the joint probability
densities are expressed  in terms of Meijer G-functions which were also found in other recent works on random matrices \cite{BeBo, For, ForKi, FoLiZi, Str}.

Given the results of \cite{AdReReSa,AkBuKiNa,IpKi} on eigenvalues, it seems to be 
natural to expect that also the singular
values of products of truncated unitary matrices have an explicit joint probability density. 
We aim at proving this statement. The squared
singular values $x_1, \ldots, x_n$ of such a product have the joint probability density
\begin{equation} \label{polensemble} 
	\frac{1}{Z_n} \prod_{j<k} (x_k-x_j) \det \left[ w_k(x_j) \right]_{j,k=1}^n, \qquad \text{ all } x_j > 0, 
\end{equation}
for certain  functions $w_k$, see Corollary \ref{corTruncations} below.

A joint probability density function of the form \eqref{polensemble} is called a polynomial ensemble in \cite{KuSt}.
It is an example of a biorthogonal ensemble \cite{Bor} whose correlation kernel is built
out of polynomials and certain dual functions. It reduces to an orthogonal polynomial
ensemble \cite{Kon} in the case $w_k(x) = x^{k-1} w(x)$ for a certain weight function $w$. 
The results of \cite{AkIpKi, AkKiWe} for the singular values of products of complex Ginibre
matrices were interpreted in \cite{KuSt} in the sense of a transformation of polynomial
ensembles.  
Recall that a complex Ginibre matrix is a random matrix whose entries are independent standard
complex Gaussians. The main result of \cite{KuSt} is the following.
\begin{theorem}[Theorem 2.1 in \cite{KuSt}] \label{Ginibreproduct} 
Let $n, l, \nu$ be non-negative integers with $1 \leq n \leq l$. 
Let $G$ be an $(n + \nu) \times l$ complex Ginibre matrix, and let $X$ be a random matrix of size 
$l \times n$, independent of $G$, such that the squared singular values $x_1, \ldots, x_n$ 
are a polynomial
ensemble 
\begin{equation} 
\propto \prod_{j<k} (x_k -x_j) \, \det \left[ f_{k}(x_j) \right]_{j,k=1}^n,  \qquad \text{ all } x_j > 0,
\end{equation}
 for certain functions $f_1, \ldots, f_n$ defined on $[0, \infty)$. 
Then the squared singular values $y_1, \ldots, y_n$ of $Y = GX$ are a polynomial ensemble
\begin{equation} 
\propto \prod_{j<k} (y_k -y_j) \, \det \left[ g_{k}(y_j) \right]_{j,k=1}^n,  \qquad \text{ all } x_j > 0,
\end{equation}
where
\begin{equation} 
g_k(y) = \int_0^{\infty} x^{\nu} e^{-x} f_k \left( \frac{y}{x} \right) \frac{dx}{x}, 
		\qquad y > 0.
\end{equation}
\end{theorem}
Note that $g_k$ is the Mellin convolution of $f_k$ with the ``Gamma density'' $x^{\nu} e^{-x}$.

We aim at an analogue of Theorem \ref{Ginibreproduct}
for the product of $X$ with a truncated unitary matrix and find that the structure of
a polynomial ensemble is preserved. Instead of a Mellin convolution with a ``Gamma density'' we
find a Mellin convolution with a ``Beta density'' $x^\nu(1-x)^\mu$ with certain variables $\nu,\mu\in\mathbb{N}_0$ and defined on the interval $[0,1]$, see Corollary \ref{TruncationTransformation}.
This result is an immediate consequence of a theorem  on the transformation of squared singular 
values of a fixed matrix $X$ when multiplied by a truncated unitary matrix that we present 
as our main theorem, see Theorem \ref{propTruncation}.

The spectral statistics of a generic truncation of a fixed matrix $X$ is an old question. Especially in random matrix theory such truncations quite often occur because of its natural relation to the Jacobi ensemble \cite{For-book} and a modification of the supersymmetry method \cite{BEKYZ} to calculate the average of an arbitrary product of characteristic polynomials where generic projections to lower dimensional matrices are needed. Applications of truncated unitary matrices in physics can be found in quantum mechanical evolution~\cite{GNJJN}, chaotic scattering~\cite{gb3}, mesoscopic physics~\cite{Been} and quantum information~\cite{ZySo}. We claim that even in telecommunications one can certainly apply truncations of unitary matrices. Usually Ginibre matrices model the information channel between the transmitter and the receiver~\cite{TV1,TV2,BM,AkKiWe}. However if the number of channels to the environment is  on the scale of the number of transmitting channels or even smaller then deviations to the Gaussian assumption should be measurable. We guess that those deviations result from the fact that no signal is lost. It can only be absorbed by the environment. Therefore the evolution in all channels should be unitary and, thus, in the channels between the receiver and the transmitter a truncated unitary matrix.

Very recently it was shown by Ipsen and one of the authors \cite{IpKi} that truncated unitary matrices are also generally involved in the spectral statistics of products of rectangular random matrices, including products of matrices which are not Ginibre matrices or truncated unitary matrices. Moreover the truncation of matrices are also important in representation theory. For example, let $H$ be an $m\times m$ Hermitian matrix and $T$  an $n\times m$ complex matrix ($m>n$) which is a truncation of a unitary matrix $U\in\U(m)$. Then the question for the generic eigenvalues of a truncation $THT^*$ ($T^*$ is the Hermitian conjugate of $T$) is deeply related to the fact which representations of $\U(n)$ are contained in a certain representation of $\U(m)$. In particular group integrals and coset integrals like integrals over Stiefel manifolds are deeply related to representation theory, see \cite{Harish,Balan1,Balan2,HPZ,CK}. Though our results are more general they can be partially interpreted in this framework. The relation of truncated unitary matrices and representation theory goes back to a parametrization of the unitary group by Gelfand and Zeitlin\footnote{We employ the transcription of Zeitlin used in \cite{CK}.} \cite{GZ}. They found that the eigenvalues of $THT^*$ with $n=m-1$ satisfies the so-called interlacing property, i.e. $x_1\leq x'_1\leq x_2\leq x'_2\leq\cdots\leq x_{m-1}\leq x'_{m-1}\leq x_m$ where $x_1\leq\cdots\leq x_m$ are the ordered eigenvalues of $H$ and $x'_1\leq\cdots\leq x'_{m-1}$ are the ordered eigenvalues of $THT^*$. This interlacing is reflected in the  ``quantum numbers'' labelling the representations of $\U(n)$ and $\U(m)$.

As a consequence of our generalization of Theorem~\ref{Ginibreproduct} to truncated unitary matrices  we derive
the joint probability density function for the squared singular values of an arbitrary product $Y = T_r \cdots T_1$,
where each $T_j$, $j=1, \ldots, r$ is a truncation of a unitary matrix. 
We find  a polynomial ensemble with Meijer G-functions which is similar to the case of products of complex Ginibre matrices \cite{AkIpKi, KuZh}. 
The polynomial ensemble is a determinantal point process with a correlation
kernel which is a double contour integral in products and ratios of Gamma functions and can be equivalently rewritten as a onefold integral over a product of two Meijer G-functions. Based on the
double integral formula we are able to obtain hard edge scaling limits
as was done in \cite{KuZh} for the product of complex Ginibre matrices. 
In addition to the Meijer G-kernels that are already in \cite{KuZh} we 
also find certain finite rank perturbations.

All results are summarized in section~\ref{sec:res}. The proofs of these results are contained in sections~\ref{sec:proofunitaryintegral},~\ref{sec:proof1},~\ref{sec:proof2}, and ~\ref{sec:proof3}. In section~\ref{conclusio} we briefly discuss open ques-tions to this topic.

\section{Statement of results}\label{sec:res}

Let us start with some preliminaries. A $k\times l$ truncation  $T$ of a matrix $U\in\U(m)$ ($m> \max(k,l)$) is a submatrix of $U$. Note that in the case that $k+l=m$ the matrix $T$ is equivalent with an element $[T]$ in the coset $\U(m)/[\U(k)\times\U(m-k)]$ due to the embedding
\begin{equation}
T\hookrightarrow[T]=\left\{\left.U\left(\begin{array}{cc} \sqrt{I_{k}-TT^*} & T \\ T^* & \sqrt{I_{m-k}-T^*T} \end{array}\right)\in\U(m)\right| U\in\U(k)\times\U(m-k)\right\},
\end{equation}
where $I_{k}$ is the $k\times k$ identity matrix. If $k+l\neq m$ such an interpretation is not possible such that our studies are more general.

A natural measure for the truncated unitary matrix $T$ is the induced Haar measure of the unitary matrix $U$. Let $dU$ be the normalized Haar measure. Then the induced measure is
\begin{equation}
\int_{\U(n)} \prod_{a=1}^k\prod_{b=1}^l\delta^{(2)}(T_{ab}-U_{ab}) dU
\end{equation}
with $\delta^{(2)}(x+i y)=\delta(x)\delta(y)$ the Dirac delta function in the complex plane.

We are interested in the singular value statistics of the matrix $Y=TX$ where $X\in\mathbb{C}^{l\times n}$. Recall that the squared singular values $y_1,\ldots, y_n$ of the matrix $Y$ are the eigenvalues of $Y^*Y$ and analogously for the matrix $X$. The joint probability density of these squared singular values  involves the Vandermonde determinant which is given by the product $\Delta(y) = \prod_{j <k}^n (y_k - y_j)$.

We present in the following subsections five results. These results comprise the joint probability distribution of the squared singular values of a product $TX$ of a fixed matrix $X$ (subsection~\ref{sec:main-res}), of a random matrix $X$ whose squared singular values are taken from a polynomial ensemble (subsection~\ref{sec:det-res}) or from a certain  Pfaffian point process (subsection~\ref{sec:Pf-res}), and of a product of truncated unitary matrices (subsection~\ref{sec:trunc-res}). Moreover we present a remarkable group integral involved in one of our proofs which is the analogue to the  Harish~Chandra/Itzykson-Zuber (HCIZ) integral \cite{Harish,ItZu} for the Ginibre ensembles, see subsection~\ref{sec:group-res}.

\subsection{Main result}\label{sec:main-res}

All corollaries we present in our work are based on the following theorem.

\begin{theorem} \label{propTruncation}
Let $n,m,l,\nu$ be non-negative integers with $1 \leq n \leq l \leq m$ and
$m \geq n + \nu + 1$. Let  $T$ be an $(n +\nu) \times l$ truncation of a Haar distributed
unitary matrix $U$ of size $m \times m$. 
Let $X$ be a non-random matrix of size $l \times n$ with non-zero squared singular values $x_1, \ldots, x_n$,
and we assume them to be all pairwise distinct.
Then the squared singular values $y_1, \ldots, y_n$ of $Y = TX$ have a joint probability density function on $[0,\infty)^n$
\begin{equation} \label{TXdensity} 
	\propto\left(\prod_{j=1}^n x_j^{-m+n} \right) \left( \prod_{j=1}^n y_j^{\nu} \right) 
	\det \left[(x_k-y_j)_+^{m-n-\nu-1} \right]_{j,k=1}^n \frac{\Delta(y)}{\Delta(x)},
	\end{equation}
where $(x-y)_+ = \max(0,x-y)$. The missing overall constant only depends on $n, m$ and $\nu$, but is independent of $X$.
\end{theorem}

We emphasize that this theorem can be easily generalized to a matrix $X$ where one or more of its squared singular values $x_j$ coincide. Then the result remains valid if we replace \eqref{TXdensity} by an appropriate limit using L'H\^opital's rule.

\begin{remark} \label{remarkln}
 We can readily set $l=n$ without loss of generality. The reason for this is exactly the same one already discussed in \cite{IpKi}. We can perform a singular value decomposition of $X=U_{\rm L}X'U_{\rm R}$ ($U_{\rm L}\in\U(l)$, $U_{\rm R}\in\U(n)$, and $X'$ a rectangular matrix with only one of its main-diagonals non-zero) and absorb the unitary matrix $U_{\rm L}$ in $T$. Since $l\geq n$ the matrix $X'$ has $l-n$ rows equal to zero. This structure projects the matrix $T$ to an even smaller matrix $T'$ of size $(n +\nu) \times n$. Let $\widetilde{X}$ be the matrix $X'U_{\rm R}$ without these zero rows, in particular $\widetilde{X}$ is $n\times n$ dimensional. Then we can consider the product $Y=TX=T'\widetilde{X}$ and apply the Theorem~\ref{propTruncation} for the fixed matrix $\widetilde{X}$ and the truncated unitary matrix $T'$. Note that $\widetilde{X}$ and $X$ have the same singular values.
\end{remark}

\subsection{An integral over the unitary group}\label{sec:group-res}

We give two proofs of Theorem \ref{propTruncation}, the first one in section \ref{sec:proof1} and the
second one in section~\ref{sec:proof2}. The first proof only works in the case where $m \geq 2n + \nu$. In this case the truncation $T$ does not have generic squared singular values which are equal to $1$. Then $T$ is sufficiently small compared to the dimension of the underlying larger unitary matrix $U$. In this case the set of all matrices $T$ which have a squared singular value equal to $1$ is a set of measure zero.

The second proof works in all cases and is based on a calculation with test functions.  
Nonetheless, we decide to keep the first proof, too, since it is based on a remarkable integral over
the unitary group that is of interest in its own right.
It replaces the HCIZ integral that was used in
the proof of Theorem \ref{Ginibreproduct}. The HCIZ integral formula \cite{Harish,ItZu} is the following well-known
integral over the unitary group $\mathcal U(n)$, 
\begin{equation} \label{HCIZ-formula} 
	\int_{\mathcal U(n)} \exp[t\Tr A U BU^*] dU = \left(\prod_{j=1}^{n-1} j! \right) 
	\frac{\det \left[ \exp(t a_j b_k) \right]_{j,k=1}^n}{t^{(n^2-n)/2}\Delta(a) \Delta(b)},
	\end{equation}
where $A$ and $B$ are Hermitian matrices with pairwise distinct eigenvalues $a_1, \ldots, a_n$, and $b_1, \ldots, b_n$,
respectively, and $t \in \mathbb{C} \setminus \{0\}$. 

The new integral over the unitary group involves the Heaviside step function of a matrix argument, defined on Hermitian matrices $X$ as
\begin{equation} \label{Heaviside-theta}
	\theta(X) = \begin{cases} 1, & \text{ if $X$ is positive definite,} \\
	0, & \text{ otherwise.} 
	\end{cases} \end{equation}
Then the generalization of \eqref{HCIZ-formula} for our purposes is the following theorem which is proven in section~\ref{sec:proofunitaryintegral}.

\begin{theorem} \label{GRunitaryintegral}
Let $A$ and $B$ be $n \times n$ Hermitian matrices with respective eigenvalues $a_1, \ldots, a_n$ and 
$b_1, \ldots, b_n$ and pairwise distinct.
Let $dU$ be the normalized Haar measure on the unitary group $\mathcal U(n)$. Then for every  $p \geq 0$,
\begin{equation} \label{Kieburg-integral} 
	\int_{\mathcal U(n)}  \det \left(A-UBU^*\right)^{p} \theta(A-UBU^*) \, dU
	= c_{n,p} \frac{\det \left[\left( a_j - b_k\right)_+^{p+n-1} \right]_{j,k=1}^n}{\Delta(a) \Delta(b)}.
	\end{equation}
The constant $c_{n,p}$ in \eqref{Kieburg-integral} is  
\begin{equation} \label{cnp} 
	c_{n,p} = \prod_{j=0}^{n-1} \binom{p+n-1}{j}^{-1}
	\end{equation}
	and $\left( a_j - b_k\right)_+$ is $\left( a_j - b_k\right)$ if the difference is positive and otherwise vanishes.
If some of the $a_j$ and/or $b_j$ coincide we have to take the formula \eqref{Kieburg-integral} 
in the limiting sense using l'H\^opital's rule.
\end{theorem}

Theorem \ref{GRunitaryintegral} is known to hold when $\min a_k  \geq \max b_j$, see \cite[formula (3.21)]{GrRi} and also \cite{HaOr}. Then the Heaviside step function drops out. However if this condition is not met it is even more surprising that the result still looks that simple, especially that the result can be expressed in exactly the same determinantal form as if the condition has been met. We emphasise that in the general case we usually do not effectively integrate over the whole group $\U(n)$ but a smaller subset. The contribution of the subset of $U \in \U(n)$ for which $A-UBU^*$ is not positive definite vanishes due to the Heaviside step function.

What does really happen in the integral~\eqref{Kieburg-integral}? To understand this we may choose $A$ and $B$ diagonal since their diagonalizing matrices can be absorbed in the unitary group integral. Recall that Harish~Chandra~\cite{Harish} traced the integral~\eqref{HCIZ-formula} back to a sum over the Weyl group acting on $B$. The Weyl group of $\U(n)$ is the permutation group of $n$ elements, $\mathcal{S}_n\subset\U(n)$. This sum results into the determinant in \eqref{HCIZ-formula}. Exactly such a thing also happens here. The difference is that the action of $\omega\in\mathcal{S}_n\subset\U(n)$ to the matrix $M=A-\omega B\omega^*$  sometimes yields no contribution to the integral because of the Heaviside step function. Because $M$ is also diagonal  the  Heaviside step function of $M$ factorizes into Heaviside step functions of $a_k-b_{\omega(k)}$ for $k=1,\ldots,n$ telling us that $a_k-b_{\omega(k)}$ has to be positive definite. Despite the fact that some of the terms in the sum over the Weyl group vanish we can extend the sum over the whole group because they are zero without changing the result. This is the reason why inside the determinant of \eqref{Kieburg-integral} we have  $\left( a_j - b_k\right)_+^{p+n-1}$ and not $\left( a_j - b_k\right)^{p+n-1}$. Hence, one can indeed understand Theorem~\ref{GRunitaryintegral} by this intuition.

\subsection{Transformation of polynomial ensemble}\label{sec:det-res}

Our main application of Theorem \ref{propTruncation} is to the situation where $X$
is random and statistically independent of $T$, in such a way that its squared singular values 
are a polynomial ensemble on $[0, \infty)$. The proof relies on 
the well-known Andr\'eief formula \cite{Andre,DeGi},
\begin{multline} \label{Andreief} 
	\int_{X^n} \det\left[ \varphi_k(x_j) \right]_{j,k=1}^n \, \det \left[ \psi_j(x_k) \right]_{j,k=1}^n 
		d\mu(x_1) \cdots d\mu(x_n) \\
	= n! \det \left[ \int_X \varphi_k(x) \psi_j(x) d\mu(x) \right]_{j,k=1}^n,
	\end{multline}
that is valid for arbitrary functions $\varphi_j$ and $\psi_k$ on a measure space $(X,\mu)$ such that the integrals exist.
The integral \eqref{Andreief} is used several times in our proofs.

With the help of Andr\'eief's formula one can readily deduce from Theorem~\ref{propTruncation} the following Corollary.

\begin{corollary} \label{TruncationTransformation}
Let $n, m, l, \nu$ and $T$ be as in Theorem \ref{propTruncation}. 
Let $X$ be a random matrix of size $l \times n$, independent of $T$,
such that the squared singular values $x_1, \ldots, x_n$ of $X$ have  the joint probability density function
\begin{equation} \label{f-pol-ensemble} 
	\propto \Delta(x) \, \det \left[ f_{k}(x_j) \right]_{j,k=1}^n,  
	\qquad \text{ all } x_j > 0, 
	\end{equation}
for certain distributions $f_1, \ldots, f_n$ such that the moments $\int_0^\infty f(x) x^a dx$ with $a=0,\ldots,n-1$ exist.
Then the squared singular values $y_1, \ldots, y_n$ of $Y = TX$ have the joint probability density
\begin{equation} \label{g-pol-ensemble} 
	\propto \Delta(y) \, \det \left[ g_{k}(y_j) \right]_{j,k=1}^n, \qquad 
	\text{ all } y_j > 0, 
	\end{equation}
where 
\begin{equation} \label{gk-definition} 
	g_k(y) = \int_0^1 x^{\nu} (1-x)^{m-n-\nu-1} f_k \left( \frac{y}{x} \right) \frac{dx}{x}, \qquad y > 0,
	\end{equation}
is the Mellin convolution of $f_k$ with the ``Beta distribution'' $x^{\nu} (1-x)^{m-n-\nu-1}$ on $[0,1]$.
\end{corollary}

We underline again that one can set $l=n$ without loss of generality.

\begin{proof}
We average \eqref{TXdensity} over $x_1, \ldots, x_n$ with the joint probability density function \eqref{f-pol-ensemble}.
The $n$-fold integral is evaluated using \eqref{Andreief} with $d\mu(x)  = x^{-m+n} dx$ on $X=[0,\infty)$ and identifying the functions $\varphi_k(x) =  f_k(x)$ and  $\psi_j(x) = (x-y_j)_+^{m-n-\nu-1}$. Then we find that the squared singular values of $Y$ have a joint probability density
\begin{equation}
\propto n! \Delta(y) \left( \prod_{j=1}^n y_j^{\nu} \right) 
	\det \left[ \int_0^{\infty} x^{-m+n} f_k(x) (x-y_j)_+^{m-n-\nu-1} dx \right]_{j,k=1}^n, 
	\end{equation}
which is of the form \eqref{g-pol-ensemble} with functions
\begin{align}
g_k(y) & = y^{\nu} \int_y^{\infty} x^{-m+n} f_k(x) (x-y)^{m-n-\nu-1} dx  \nonumber\\
	& =  y^{\nu} \int_0^1 \left(\frac{y}{x}\right)^{-m+n} \left(\frac{y}{x} - y\right)^{m-n-\nu-1} 
		f_k\left(\frac{y}{x}\right)  \frac{y dx}{x^2},
\end{align}
Here, we applied the change of variables $x \mapsto y/x$. This 
is easily seen to reduce to \eqref{gk-definition}. 
 \hfill{$\Box$}
\end{proof}			
Corollary \ref{TruncationTransformation} was obtained in the recent preprint \cite{Kui} in a different way.

Corollary \ref{TruncationTransformation} is the analogue of Theorem \ref{Ginibreproduct}
for the case of a multiplication with truncated unitary matrix. It is interesting to
note that Theorem \ref{Ginibreproduct} is obtained from Corollary \ref{TruncationTransformation}
in the limit when $m \to \infty$ while keeping $n$, $l$, and $\nu$ fixed, since $\sqrt{m} T \to G$ where $G$ is
a complex Ginibre matrix. Recall that $m$ is the dimension of the unitary matrix that
$T$ is a truncation of.

\subsection{Transformation of Pfaffian ensembles}\label{sec:Pf-res}

Theorem \ref{propTruncation} can also be applied in a Pfaffian context.
Instead of \eqref{Andreief}  we now use the de Bruijn formula \cite{DeB}, see \cite[Proposition 6.3.5]{For-book}
\begin{multline} \label{DeBruijn}
	\int_{X^n} \det\left[ \varphi_k(x_j) \right]_{j,k=1}^n \, \Pf \left[ f(x_j,x_k) \right]_{j,k=1}^n 
		d\mu(x_1) \cdots d\mu(x_n) \\
	= n! \Pf \left[ \int_X \int_X \varphi_j(x) \varphi_k(y) f(x,y) d\mu(x) d\mu(y) \right]_{j,k=1}^n,
	\end{multline}
which is valid for $n$ even and  $f$ an anti-symmetric function on $X \times X$, i.e., $f(y,x) = -f(x,y)$
for all $x,y \in X$, such that all integrals exist. 
Then the following result is a consequence of Theorem \ref{propTruncation}  which can be deduced in a similar way as we obtained
Corollary \ref{TruncationTransformation}. Therefore we skip the proof and only state the Corollary.

\begin{corollary} \label{PfaffianTransform}
Let $n,m,l, \nu$ and $T$ be as in Theorem \ref{propTruncation}
with $n$ even.
Let $X$ be independent of $T$ such that the squared singular values
of $X$ have the joint probability density function
\begin{equation}
 \propto \Delta(x) \Pf \left[ f(x_j, x_k) \right]_{j,k=1}^n, \qquad \text{ all } x_j > 0,
\end{equation}
for a certain anti-symmetric distribution $f$ on $[0,\infty) \times [0,\infty)$ such that the mixed moments $\int_0^{\infty}\int_0^\infty x_1^jx_2^k f(x_1,x_2)dx_1dx_2$ exist for all $j,k=0,\ldots,n-1$.
Then the squared singular values $y_1, \ldots, y_n$ of $Y = TX$ have joint
probability density
\begin{equation}
  \propto \Delta(y) \Pf \left[ g(y_j,y_k) \right]_{j,k=1}^n, \qquad \text{ all } y_j > 0,
\end{equation}
where 
\begin{equation}
g(y_1, y_2) = \int_0^1 \int_0^1 x_1^{\nu} (1-x_1)^{m-n-\nu-1} x_2^{\nu} (1-x_2)^{m-n-\nu-1}
	f\left( \frac{y_1}{x_1}, \frac{y_2}{x_2}\right) \frac{dx_1}{x_1} \frac{dx_2}{x_2}.
\end{equation}
\end{corollary}

This result can be combined with Corollary~\ref{TruncationTransformation} for ensembles where we have a mixture of orthogonal and skew-orthogonal polynomials, i.e., the joint probability density of the squared singular values of $X$ is given by 
\begin{equation}\label{mix}
 \propto \Delta(x) \Pf \left[\begin{array}{cc} f(x_j, x_k) & f'_i(x_j) \\ -f'_i(x_j) & 0 \end{array}\right]_{\substack{j,k=1,\ldots,n\\ i=1,\ldots,n'}}, 
 \qquad \text {all } x_j > 0,
\end{equation}
where $n+n'$ is even and $f'_i$ is an additional set of distributions. Also this structure carries over to the product $TX$ as can be easily checked.

The structure~\eqref{mix} is not only academically. It appears if $X$ is real and $n$ is odd (in this case we have $n'=1$), see for example the real Laguerre ensemble \cite{Methabook} or the real Jacobi ensemble \cite{For-book}. Also the case $n'>1$ appears naturally in random matrix theory. For applications in QCD a random matrix model was proposed which breaks the Gaussian unitary ensemble by the chiral Gaussian unitary ensemble \cite{DSV}. The joint probability density of this random matrix ensemble has the form~\eqref{mix}, see \cite{AN,Kieburg}.

\subsection{Products of truncated unitary matrices}\label{sec:trunc-res}

Let us come to the product of a finite number of truncated unitary matrices. The joint probability density of the squared singular values readily follows from Corollary \ref{TruncationTransformation}. As in the case of multiplying Ginibre matrices the Meijer G-functions play a crucial role in the spectral statistics.
Meijer G-functions are defined via a contour integral
 \begin{multline} \label{MeijerG-def}
G_{p,q}^{m,n} \left( \left. {a_1,\dots, a_n;a_{n+1}, \ldots,a_p \atop b_1,\dots, b_m;b_{m+1}, \ldots,b_q} \right|z\right) \\
= \frac{1}{2\pi i}
\int_C {\prod\limits_{j=1}^m \Gamma ( b_j + s) \prod\limits_{j=1}^n \Gamma(1 - a_j - s) \over
\prod\limits_{j=m+1}^q \Gamma (1-  b_j - s) \prod\limits_{j=n+1}^p \Gamma(a_j + s)}  z^ {-s} ds, 
\end{multline}
where the contour $C$ separates the poles of $\prod_{j=1}^m \Gamma(b_j+s)$ from the poles of  $\prod_{j=1}^n \Gamma(1-a_j-s)$, see \cite{BeSz,Luk} for more details.

We apply Corollary \ref{TruncationTransformation} to the product of $r$ truncated
unitary matrices. 
For $j=1, \ldots, r$, let $T_j$ be a matrix of size $(n + \nu_j) \times (n + \nu_{j-1})$  
where $\nu_0 = 0$ and $\nu_1, \ldots, \nu_r$ are non-negative integers.
Suppose $T_j$ is the truncation of a Haar distributed unitary matrix of size $m_j \times m_j$
where $m_1 \geq 2n + \nu_1$ and $m_j \geq n+ \nu_j+1$ for $j=2, \ldots, r$. The
squared singular values of $T_1$ have the joint probability density \eqref{Jacobi-ensemble}
with parameters $n_1$, $m_1$ and $\nu_1$. This is a polynomial ensemble~\cite{For-book}
\begin{equation} \label{Jacobiensemble}
 \frac{1}{Z_n} \Delta(x) \det \left[ w_k^{(1)}(x_j) \right]_{j,k=1}^n 
	\end{equation} 
with 
\begin{equation} \label{defwk1} 
	w_k^{(1)}(x) = \begin{cases} x^{\nu_1+k-1} (1-x)^{m_1-2n-\nu_1}, & \text{ if } 0 < x < 1, \\
	0, & \text{otherwise}. \end{cases} 
	\end{equation} 
	The constant $Z_n$ normalizes the joint probability density \eqref{Jacobiensemble}.
	Then we find the following Corollary.

\begin{corollary} \label{corTruncations}
Let $Y = T_r \cdots T_1$ with truncated unitary matrices $T_j$ as described above.
Then the squared singular values of $Y$ have the joint probability density
\begin{equation} \label{wkr-ensemble} 
	\frac{1}{Z_n^{(r)}} \, \prod_{j < k} (y_k-y_j) \, \det \left[ w_k^{(r)}(y_j) \right]_{j,k=1}^n ,
	\end{equation}
where $w_k^{(r)}$ is given by \eqref{defwk1} in the case $r=1$ and by
\begin{equation} \label{defwkr} 
	w_k^{(j)}(y) = \int_0^1 x^{\nu_j}(1-x)^{m_j-n-\nu_j-1} w_k^{(j-1)} \left( \frac{y}{x}\right) \frac{dx}{x}
	\end{equation}
for $j=2, \ldots, r$ when $r \geq 2$.
\end{corollary}
The weight functions $w_k^{(r)}$ are obtained as an $(r-1)$-fold Mellin convolution of the
``Beta distribution". The function $w_k^{(1)}$ can be written as a Meijer G-function,
\begin{equation} \label{wk1MeijerG}
w_k^{(1)}(x) = c_{1}  \, \MeijerG{1}{0}{1}{1}{m_1-2n +k}{\nu_1 +k-1}{x}
	\end{equation}
with $c_{1} = \Gamma(m_1-2n-\nu_1+1)$. Since the 
class of Meijer G-functions is closed under the Mellin convolution, see \cite[formula (5)]{BeSz},
we find from \eqref{defwkr} and \eqref{wk1MeijerG},
\begin{align} \nonumber 
	w_k^{(r)}(y) & = c_{r} \MeijerG{r}{0}{r}{r}{m_r-n, \ldots, m_2-n, \ldots, m_1-2n+k}{\nu_r, \nu_{r-1}, \ldots,\nu_2, \nu_1 + k-1}{y} \\
		& = \frac{c_r}{2\pi i} \int_{C} 
			\frac{\Gamma(\nu_1+k-1+s) \prod_{j=2}^r \Gamma(\nu_j + s)}{\Gamma(m_1-2n+k+s) \prod_{j=2}^r \Gamma(m_j-n+s) } y^{-s} ds,
			\quad 0 < y < 1,
			\label{wkr-integral}
			\end{align}
Here the contour $C$  is a positively oriented curve in the complex $s$-plane that starts and ends at $-\infty$ and 
encircles the negative real axis. The constant $c_r$  in \eqref{wkr-integral} is
\begin{equation} 
c_{r} = \Gamma(m_1-2n-\nu_1+1) \prod_{j=2}^r \Gamma(m_j-n-\nu_j).
	\end{equation}

It can be checked from \eqref{wkr-integral} that the linear span of $w_1^{(r)}, \ldots, w_n^{(r)}$ consists of all functions of the form
\begin{equation} \label{wkr-span} 
	\frac{1}{2\pi i} \int_{C} \frac{q(s) \prod_{j=1}^r \Gamma(s+\nu_j)}{\prod_{j=1}^r \Gamma(s+m_j-n)}  y^{-s} ds, \qquad
	0 < y < 1, 
	\end{equation}
where $q(s)$ is a polynomial of degree smaller than $n$. Remarkably enough, this space does 
not depend on the ordering
of the parameters $m_1, \ldots, m_r$ and neither on the ordering of the parameters $\nu_1, \ldots, \nu_r$. 

The $k$-point correlation function of the joint probability density \eqref{wkr-ensemble} satisfies a determinantal point process on $[0,1]$. In the center of such a determinantal point process stands a correlation kernel that can always be written as
\begin{equation} \label{Knsum} 
	K_n(x,y) = \sum_{k=0}^{n-1} P_k(x) Q_k(y).
	\end{equation}
For a polynomial ensemble, the function $P_k$ is a polynomial of degree $k$ and $Q_k$ is in the linear span of $w_1^{(r)}, \ldots, w_n^{(r)}$
satisfying the biorthogonality 
\begin{equation} \label{PjQk-biorthogonal} 
	\int_0^1 P_j(x) Q_k(x) dx = \delta_{j,k},
	\end{equation}
see e.g.~\cite{Bor}.
As in \cite{KuZh} we find integral representations for the biorthogonal functions
$P_k$ and $Q_k$ and a double integral formula for $K_n$.
In what follows we use the Pochhammer symbol
\begin{equation}
(a)_k = a (a+1) \cdots (a+k-1) = \frac{\Gamma(a+k)}{\Gamma(a)}.
\end{equation}

\begin{proposition} \label{PkQkintegrals}
For every $k=0, \ldots, n-1$, we have
\begin{align} 
 P_k(x) & = \frac{1}{2\pi i} \oint_{\Sigma_k} 
	\frac{1}{(t-k)_{k+1}} \prod_{j=1}^r  \frac{\Gamma(t+1+ m_j-n)}{\Gamma(t+1+\nu_j)} \,  x^t \, dt \nonumber\\
	&=G_{r+1,r+1}^{0,r+1} \left( \left. {k+1,n-m_1,\ldots,n-m_r \atop 0,-\nu_1, \ldots,-\nu_r} \right|x\right) ,\label{defPk}
\end{align}
where $\Sigma_k$ is a closed contour encircling the interval $[0,k]$ once in positive direction
and is not encircling any pole of the integrand in \eqref{defPk} that is outside $[0,k]$. Moreover we have
\begin{align}
 Q_k(y) & =
	 \begin{cases} \ds
	\frac{1}{2\pi i} \int_{C} (s-k)_k \prod_{j=1}^r \frac{\Gamma(s+\nu_j)}{\Gamma(s+m_j-n)} \, y^{-s} ds,
		& 0 < y < 1, \\
		0, & \text{elsewhere} \end{cases}\nonumber\\
		&= G_{r+1,r+1}^{r+1,0} \left( \left. {-k,m_1-n,\ldots,m_r -n\atop 0,\nu_1, \ldots,\nu_r} \right|y\right) \label{defQk}
\end{align}
with the same contour $C$ as in \eqref{wkr-integral}.

The kernel $K_n$ from \eqref{Knsum} is
\begin{align} 
	K_n(x,y) &= 
\frac{1}{(2\pi i)^2}
		\int_{C} ds \oint_{\Sigma_n} dt \prod_{j=0}^r \frac{\Gamma(s+1+\nu_j) \Gamma(t+1+m_j-n)}{ \Gamma(t+1+\nu_j) \Gamma(s+1+m_j-n)}   
		\frac{x^t y^{-s-1}}{s-t} \nonumber\\
		&= -\int_0^1 G_{r+1,r+1}^{0,r+1} \left( \left. {n-m_0,\ldots,n-m_r \atop -\nu_0, \ldots,-\nu_r} \right|ux\right)\nonumber\\
		&\qquad \qquad \qquad \times G_{r+1,r+1}^{r+1,0} \left( \left. {m_0-n,\ldots,m_r -n\atop \nu_0, \ldots,\nu_r} \right|uy\right)du\label{Knintegral} 
\end{align}
which is valid if $\Sigma_n$ and $C$ do not intersect. In \eqref{Knintegral} it is understood that $m_0 = \nu_0 =0$.
\end{proposition}
The proposition is proved in section \ref{sec:proof3}.

\subsection{Hard edge scaling limits}

Based on the double integral representation \eqref{Knintegral} we analyze some scaling limits of the correlation kernel as $n \to \infty$.
In a forthcoming publication we show that the usual sine kernel limit can be found in the bulk and the Airy kernel limit at the soft edges. These results will be reported elsewhere, see also \cite{LiWaZh}.

Here we want to look at a more straightforward scaling limit, namely the hard edge scaling at the origin in the following situation.
Taking $n \to \infty$, we simultaneously have to let $m_j \to \infty$ for $j=1, \ldots, r$,
since $m_1 \geq 2n + \nu_1$ and $m_j \geq n + \nu_j + 1$ for $j \geq 2$. 
We keep $\nu_j$ fixed for every $j=1, \ldots, r$.
We choose a subset $J$ of indices
\begin{equation} \label{Jset}
J = \{j_1, \ldots, j_q \} \subset \{2, \ldots, r\}, \qquad 
	\text{with } 0 \leq q = |J|  < r
\end{equation}
and integers $\mu_1, \ldots, \mu_q$ with $ \mu_k \geq \nu_{j_k} + 1$,
and we assume
\begin{align*}
	m_j - n & \to \infty && \text{ for } j \in \{1, \ldots, r\} \setminus J,  \\
	m_{j_k} - n &= \mu_k \text{ is constant} && \text{ for } j = j_k \in J. 
	\end{align*}
This leads to our final result, which we also prove in section \ref{sec:proof3}.
\begin{theorem} \label{Knlimit}
In the above setting,  we put 
\begin{equation} \label{defcn} 
	c_n = n \prod_{j \not\in J}  (m_j-n). 
	\end{equation}
Then the kernels $K_n$ from \eqref{Knintegral} have the following hard edge scaling limit,
\begin{align} \nonumber 
	\lim_{n \to \infty} & \frac{1}{c_n} K_n \left( \frac{x}{c_n}, \frac{y}{c_n} \right) \\
	& = 
	\frac{1}{(2\pi i)^2} \int_{-\frac{1}{2}-i\infty}^{-\frac{1}{2} + i \infty} ds \int_{\Sigma} dt
		\prod_{j=0}^{r}  \frac{\Gamma(s+1+\nu_j)}{\Gamma(t+1+\nu_j)} \frac{\sin \pi s}{\sin \pi t} 
		\prod_{k=1}^q \frac{\Gamma(t+1+\mu_k)}{\Gamma(s+1+\mu_k)}
			\frac{x^t y^{-s-1}}{s-t} \nonumber \\
	& =  -\int_0^1 G_{q,r+1}^{1,q} \left( \left. {-\mu_{1},\ldots,-\mu_q \atop 0;-\nu_1, \ldots,-\nu_r} \right|ux\right)  
	G_{q,r+1}^{r,0} \left( \left. {\mu_{1},\ldots,\mu_q\atop \nu_1, \ldots,\nu_r;0} \right|uy\right)du, \label{scalinglimit}
\end{align} 
		where $\Sigma$ is a contour around the positive real axis in the half-plane $\Re t > - \frac{1}{2}$.
\end{theorem}

The kernel \eqref{scalinglimit} reduces to the Meijer G-kernel described in \cite{KuZh} in case $q=0$. These kernels appeared before for limits of products of Ginibre matrices \cite{KuZh}, products with inverses of Ginibre matrices \cite{For14},
for biorthogonal ensembles \cite{KuSt}, and also as limits for Cauchy two matrix models \cite{BeBo,BeGeSz,ForKi}.
According to Theorem \ref{Knlimit} we obtain the same limits for products
of truncated unitary matrices provided that the dimensions $m_j$ of the
underlying unitary matrices become large compared to $n$, in the sense that
$m_j - n \to +\infty$ for every $j$.

The kernels  \eqref{scalinglimit} are new for $q \geq 1$, and these
are finite rank perturbation of the Meijer G-kernels from \cite{KuZh}. To see this,
we recall that $\mu_k \geq \nu_{j_k}+1$. Let us assume, for notational simplicity,
that $j_k = r-q+k$ in \eqref{Jset}. Then
\begin{equation} \label{defRt} 
	R(t) := \prod_{k=1}^q \frac{\Gamma(t+1+\mu_k)}{\Gamma(t+1+\nu_{r-q+k})}
	\end{equation}
is a polynomial of degree $\deg R = \sum_{k=1}^q (\mu_k - \nu_{r-q+k})$, and \eqref{scalinglimit}  can be written as
\begin{equation}
	\frac{1}{(2\pi i)^2} \int_{-\frac{1}{2}-i\infty}^{-\frac{1}{2} + i \infty} ds \int_{\Sigma} dt
		\prod_{j=0}^{r-q}  \frac{\Gamma(s+1+\nu_j)}{\Gamma(t+1+\nu_j)} \frac{\sin \pi s}{\sin \pi t} \frac{R(t)}{R(s)}
			\frac{x^t y^{-s-1}}{s-t}.
\end{equation}
This is indeed a finite rank perturbation of the Meijer G-kernel with
parameters $\nu_1, \ldots, \nu_{r-q}$, since $R$ is a polynomial.
In particular for $q = r- 1$, it is a finite rank modification of the hard edge Bessel kernel.
Such finite rank modifications were also obtained in \cite{DeFo} in a somewhat different context.

In the ensuing sections we prove our statements. 
We start in section \ref{sec:proofunitaryintegral} with the proof of Theorem \ref{GRunitaryintegral}
since it  is used in the first proof of Theorem \ref{propTruncation} that we present
in section \ref{sec:proof1}.
The second proof is in section \ref{sec:proof2}. This proof is a rather lenghty sequence
of matrix integral evaluations and we have broken it up into six steps.
The proofs of Proposition \ref{PkQkintegrals} and Theorem \ref{Knlimit} are shown in section~\ref{sec:proof3}.

\section{Proof of Theorem \ref{GRunitaryintegral}} \label{sec:proofunitaryintegral}

For the proof of Theorem \ref{GRunitaryintegral}, we need the Ingham-Siegel formula \cite{Ingham,Siegel}
\begin{equation} \label{Ingham-Siegel} 
	\int_{{\mathcal H}(n)} \exp[i \Tr HX] \det(H-zI_n)^{-n-p} dH = c \exp[iz \Tr X] \det X^p \, \theta(X)
	\end{equation}
with a normalization constant $c$ depending on $p$ and $n$, only, which can be fixed by the choice $X=I_n$ and $z=i$. This integral  is valid for Hermitian matrices $X\in{\mathcal H}(n)$ and $\Im z > 0$.
The integral is over the space $\mathcal H(n) =  \mathfrak{gl}(n)\slash\mathfrak{u}(n)$ of $n \times n$ Hermitian matrices $H$ with the flat Lebesgue measure,
\begin{equation} 
dH = \prod_{j=1}^n dH_{jj} \, \prod_{j < k} d \Re H_{jk} \, d\Im H_{jk}.
	\end{equation}
Here $\mathfrak{gl}(n)$ and $\mathfrak{u}(n)$ are the Lie algebras of the  general linear and the unitary group, respectively.
If $p$ is not an integer then we define via the spectral representation,
\begin{equation} 
\det(H-zI_n)^{-n-p} = \prod_{j=1}^n (h_j - z)^{-n-p},
	\end{equation}
where $h_1, \ldots, h_n$ are the real eigenvalues of the Hermitian matrix $H$ and $(h-z)^{-n-p}$
is defined in the complex $h$-plane with a  branch cut along $\{ z + iy \mid y \geq 0 \}$,
and with $h^{n+p} (h-z)^{-n-p} \to 1$ as $h \to +\infty$.

For $n=1$ the Ingham-Siegel formula \eqref{Ingham-Siegel}  reduces to 
\begin{equation} \label{Ingham-Siegel-nis1} 
	\int_{-\infty}^{\infty} \frac{e^{i xs}}{(s-z)^{1+p}} ds = 
\begin{cases} \ds \frac{2\pi i e^{\pi i p/2}}{\Gamma(p+1)} e^{izx} x^p,  & \text{ if } x \geq 0, \\ 0, & \text{ if } x < 0,
\end{cases} \end{equation} 
which is valid for $\Im z > 0$ and $p \geq 0$. It can be verified
by contour integration.

\begin{proof}[Proof of Theorem \ref{GRunitaryintegral}]
We consider the integral
\begin{equation} \label{integral-def} 
	J(A,B)=\int_{\mathcal U(n)}  \det \left(A-UBU^*\right)^{p} \theta(A-UBU^*) \, dU.
	\end{equation}
The determinant can be rewritten via the Ingham-Siegel formula \eqref{Ingham-Siegel} identifying $X = A - U B U^*$. We obtain
\begin{equation} \label{Kieburg1}	
	J(A,B)\propto e^{-iz \Tr (A-B)} \int_{\mathcal U(n)}  \int_{\mathcal H(n)} e^{i \Tr H(A-UBU^*)} \det(H-zI_n)^{-n-p} dH dU 
	\end{equation}
Both integrals are absolutely integrable  because the integral over $U$ is over a compact set with a continuous integrand and the integral over $H$ is bounded by $|\Delta(h)\prod_{j=1}^n (h_j-z)^{-n-p}/\Delta(a)|$ for $\sum_{j=1}^n h_j^2\to\infty$. Recall that $a_1,\ldots,a_n$ are the real, pairwise distinct eigenvalues of $A$. Hence we can interchange the order of integration.

The integral over the unitary group is 
evaluated with the HCIZ formula \eqref{HCIZ-formula}.  
Then \eqref{Kieburg1} is
up to a constant
\begin{equation} \label{Kieburg2} 
	J(A,B)\propto  e^{-iz\Tr(A-B)} \int_{\mathcal H(n)} 
	e^{i \Tr HA} \det(H-zI_n)^{-n-p} \frac{\det\left[ e^{-i h_j b_k}\right]}{\Delta(h) \Delta(b)} dH.
	\end{equation}
We write $H = V h V^*$ for the eigenvalue decomposition of $H$ where $V\in\U(n)$ and $h=\diag(h_1,\ldots,h_n)$. Moreover we use that
$dH\propto\Delta(h)^2 \, dV dh_1 \ldots dh_n$, see e.g.\ \cite{Dei}. Using this in \eqref{Kieburg2} leads to
\begin{align}  
	J(A,B)&\propto  \frac{e^{-iz\Tr(A-B)}}{\Delta(b)} \int_{\mathbb R^n} \int_{\mathcal U(n)} e^{i \Tr V h V^* A}  
	\prod_{j=1}^n (h_j - z)^{-n-p}\nonumber\\
	& \qquad \qquad \times\Delta(h) 
	\det\left[ e^{-i h_j b_k}\right]_{j,k=1}^n dV dh_1 \cdots dh_n \label{Kieburg3}
	\end{align}
The integral over $V \in \mathcal U(n)$ is again a HCIZ integral \eqref{HCIZ-formula},
\begin{align} 
	J(A,B)&\propto  \frac{e^{-iz\Tr(A-B)}}{\Delta(a) \Delta(b)} \int_{\mathbb R^n} \det \left[e^{i h_j a_k} \right]_{j,k=1}^n
	\prod_{j=1}^n (\lambda_j-z)^{-n-p} \nonumber\\
	&\qquad \qquad \times\det\left[e^{-i h_j b_k}\right]_{j,k=1}^n  dh_1 \cdots dh_n.\label{Kieburg4} 
	\end{align}
The factors in the product $\prod_{j=1}^n (\lambda_j-z)^{-n-p}$ can be pulled into either one of the
two determinants. Then, Andr\'eief's identity~\eqref{Andreief} can be applied to find
\begin{equation} \label{Kieburg5} 
	J(A,B)\propto  \frac{e^{-iz\Tr(A-B)}}{\Delta(a) \Delta(b)} 
	\det \left[ \int_{-\infty}^{\infty}   \frac{e^{i s (a_j-b_k)}}{(s - z)^{n+p}} 	ds \right]_{j,k=1}^n.
	\end{equation}
The integral in the determinant is of the form \eqref{Ingham-Siegel-nis1} which is 
up to a constant equal to $e^{iz(a_j-b_k)} (a_j-b_k)_+^{n+p-1}$.
The exponential factors $e^{iz(a_j-b_k)}$ inside the determinant cancel with those in front of the determinant. 
The resulting expression is
\begin{equation} \label{Kieburg6}
	 J(A,B)\propto  \frac{ \det\left[ (a_j-b_k)_+^{n+p-1} \right]}{\Delta(a) \Delta(b)},
	 \end{equation}
	which is up to a constant indeed the right hand side of \eqref{Kieburg-integral}.
	
Whenever $A - UBU^*$ is positive definite for all $U\in\U(n)$, the formula \eqref{Kieburg-integral} reduces to 
\begin{equation} \label{GrossRichards-integral} 
	\int_{\mathcal U(n)}  \det \left(A-UBU^*\right)^{p} dU
	= c_{n,p} \frac{\det \left[\left( a_j - b_k\right)^{p+n-1} \right]_{j,k=1}^n}{\Delta(a) \Delta(b)}.
	\end{equation}
This is equivalent to an integral given by Gross and Richards in \cite[formula (3.21)]{GrRi}, namely
\begin{equation} \label{GrossRichards-integral2} 
 \frac{\det \left[\left(1-s_j t_k\right)^{-a} \right]_{j,k=1}^n}{\Delta(s) \Delta(t)} = 
	\tilde{c}_{n,a} \int_{\mathcal U(n)}  \det \left(I_n-SUTU^*\right)^{-(a+n-1)} dU
	\end{equation}
	with $\tilde{c}_{n,a} = \prod_{j=0}^{n-1} (a)_j/j!$.
The formula \eqref{GrossRichards-integral2} is valid whenever $S$ and $T$ are Hermitian matrices
with eigenvalues $s_1, \ldots, s_n$ and $t_1, \ldots, t_n$, respectively, satisfying $|s_j t_k|< 1$ for all
$j,k=1, \ldots,n$. 
The formulas \eqref{GrossRichards-integral} and \eqref{GrossRichards-integral2} are related if we
take $-a = p+n-1$, $S = A^{-1}$ and $T=B$. The constants are related by $c_{n,p} = 1/\tilde{c}_{n,a}$,
and the formula \eqref{cnp} follows which completes the proof of Theorem \ref{GRunitaryintegral}.
 \hfill{$\Box$}
\end{proof}

\section{First proof of Theorem \ref{propTruncation}} \label{sec:proof1}

As already said, our first proof of Theorem \ref{propTruncation} only
works if $m \geq 2n +\nu$. 
In that case there is an explicit formula for the distribution
of a truncation $T$ of size $(n +\nu) \times n$, namely
\begin{equation} \label{truncated-distribution} 
	c \, \det(I _n- T^*T)^{m-2n-\nu} \theta(I_n-T^*T) \, dT 
\end{equation}
with $dT = \prod_{j=1}^{n+\nu} \prod_{k=1}^n d \Re T_{jk} \, d \Im T_{jk}$ the flat Lebesgue measure on the space of $(n +\nu) \times n$ rectangular complex matrices and $c$ 
is a constant. The function $\theta$ is  the
Heaviside step function of a matrix argument defined in \eqref{Heaviside-theta}.
The (unordered) eigenvalues $t_1, \ldots, t_n$ of $T^*T$ are in the interval $[0,1]$ and have the joint
probability density 
\begin{equation} \label{Jacobi-ensemble} 
	\frac{1}{Z_n} \prod_{j<k} (t_k-t_j)^2 \prod_{j=1}^n t_j^{\nu} (1-t_j)^{m-2n-\nu},
	\qquad 0 \leq t_1, \ldots, t_n \leq 1. 
	\end{equation}
This is an example of a Jacobi ensemble \cite{Col}.
When $m < 2n +\nu$, then $T^*T$ has the eigenvalue $1$ with multiplicity of at least $2n+\nu - m$ and
the density \eqref{truncated-distribution} is not valid anymore.

We follow the proof of Lemma 2.2 in \cite{KuSt} except that at a certain stage in
the proof the HCIZ integral \eqref{HCIZ-formula}  is replaced by the
integral \eqref{Kieburg-integral}. 

\begin{proof}[Proof of Theorem \ref{propTruncation} in the case $m \geq 2n+\nu$] \quad
As already discussed in Remark \ref{remarkln} we may restrict to the case $l = n$ without loss of generality.

Consider a fixed square matrix $X$ of size $n \times n$ which is assumed to be invertible. The change of variables $T \mapsto Y = TX$ has the Jacobian, see e.g.\ \cite[Theorem 3.2]{Mat},
\begin{equation}
\det (X^* X)^{-(n+\nu)} = \prod_{k=1}^n x_k^{-n-\nu}.
\end{equation}
The distribution \eqref{truncated-distribution} on $T$ (where $T$ has size $(n+\nu) \times n$) 
then transforms into the distribution 
\begin{align} 
 & \propto \prod_{k=1}^n x_k^{-n-\nu} \, \det(I_n - (X^{-1})^* Y^* Y X^{-1})^{m-2n-\nu}\nonumber 
  \theta(I_n-(X^{-1})^* Y^* Y X^{-1}) \, dY\\
 & = \prod_{k=1}^n x_k^{-m+n} \, \det(X^* X -  Y^* Y)^{m-2n-\nu} \theta(X^*X - Y^* Y) \, dY. \label{Y-distribution}
	\end{align}

In the next step we perform a singular value decomposition $Y = V \Sigma U$ with Jacobian \cite{EdRa}
\begin{equation}
dY \propto \left( \prod_{j=1}^n y_j^{\nu} \right)	\Delta(y)^2 dU dV dy_1 \ldots dy_n
\end{equation}
written in terms of the squared singular values $y_1, \ldots, y_n$ of $Y$. The measure $dU$ is the Haar measure on $\U(n)$ and $dV$ is the
invariant measure on $\U(n+\nu)\slash[\U^n(1)\times\U(\nu)]$. We use this fact in \eqref{Y-distribution} to perform the integration of $V$,
which only contributes to the constant.
This yields a probability measure on $\mathcal U(n) \times [0,\infty)^n$ proportional to
\begin{multline}
 \propto\left( \prod_{k=1}^n x_k^{-m+n} \right) \left( \prod_{j=1}^n y_j^{\nu} \right) \, 
	\det(X^* X - U^* \Sigma^2 U)^{m-2n-\nu} \\ 
	\times\theta(X^*X - U^*\Sigma^2 U)\Delta(y)^2 \, dU  dy_1 \cdots dy_n.  \label{Uy-distribution} 
\end{multline}
	
The integral over $U$ in \eqref{Uy-distribution} can be done with the help of the 
integral \eqref{Kieburg-integral} with $A = X^*X$, $B = \Sigma^2$, and $p = m -2n - \nu$.
This leads to the density for the squared singular values $y_1, \ldots, y_n$ of $Y$,
given those of $X$, which is proportional to \eqref{TXdensity}  and Theorem \ref{propTruncation}
follows  for $m \geq 2n +\nu$.  \hfill{$\Box$}
\end{proof}

\section{Second proof of Theorem \ref{propTruncation}} \label{sec:proof2}

We underline that the second approach to prove Theorem \ref{propTruncation} does not rely on the restriction $m\geq 2n+\nu$.
As before we denote the set of $n \times n$ unitary matrices and of $n \times n$ Hermitian matrices by $\mathcal U(n)$ and $\mathcal H(n)$, respectively. We also
use $\mathcal M(m,n)$ for the space of $m \times n$ complex matrices and abbreviate $\mathcal M(m) = \mathcal M(m,m)$.

Also in the second approach we assume that $l = n$ because it does not restrict generality,
see Remark \ref{remarkln}.
We assume $X$ to be a fixed $n \times n$ matrix with non-zero squared singular values.

\subsection{Preliminaries}

Let $f$ be a symmetric function in $n$ variables. We extend $f$ to Hermitian matrices $A$ by defining
$f(A) = f(a_1, \ldots, a_n)$ if $a_1, \ldots, a_n$ are the eigenvalues of $A$. 
With a slight abuse of notation we also define $f(B)$ for  $(n+\nu) \times (n+\nu)$ matrices $B$ having
$\nu$ eigenvalues equal to $0$, by putting
$f(B) = f(b_1, \ldots, b_n)$ if $b_1, \ldots, b_n$ are the non-zero eigenvalues of $B$.

Then our aim is to prove that for all continuous symmetric functions $f$ on $[0,\infty)^n$,
we have
\begin{equation} \label{toprove1} \mathbb E \left[f(Y^*Y) \right] =  
	 \int_{[0,\infty)^n} f(y_1, \ldots, y_n) p_{X,Y}(x,y) dy_1 \ldots dy_n 
	\end{equation}
where, for a given $X$, $y \mapsto p_{X,Y}(x,y)$ denotes the density from \eqref{TXdensity}.
It will be enough to prove \eqref{toprove1} for symmetric polynomial functions $f$,
since the density $p_{X,Y}$ has a compact support and the symmetric polynomials
are then obviously uniformly dense in the set of all continuous symmetric functions.

Note that, by our definition of $f$ on matrices, we have $f(Y^* Y) = f(Y Y^*) = f(TX X^* T^*)$.
Since $T$ is the truncation of a Haar distributed unitary matrix $U$ of size $m \times m$, we have
\begin{align}
 \mathbb E \left[f(YY^*) \right]&=\int_{\mathcal U(m)} f \left(\begin{pmatrix} I_{n+\nu} & O_{n+\nu,m-n-\nu}\end{pmatrix}\widetilde{U}
		XX^* \widetilde{U}^*\begin{pmatrix} I_{n+\nu} \\ O_{m-n-\nu,n+\nu}\end{pmatrix}
			\right) dU \nonumber\\
			&\propto \int_{\mathcal M(m,n)} f \left(\begin{pmatrix} I_{n+\nu} & O_{n+\nu,m-n-\nu}\end{pmatrix}M
		XX^* M^*\begin{pmatrix} I_{n+\nu} \\ O_{m-n-\nu,n+\nu}\end{pmatrix}
			\right) \nonumber\\
			&\qquad \times\prod\limits_{1\leq j<k\leq n} \delta^{(2)}(\{M^*M\}_{jk})\prod_{j=1}^{n} \delta(\{M^*M\}_{jj}-1) dM, \label{EfYYstar}
\end{align}
where $dU$ is the normalized Haar measure on the unitary group $\U(m)$. The matrix $O_{p,q}$ is the zero matrix of  size $p \times q$. The complex $m\times n$ matrix  
\begin{equation}
\widetilde{U} = U \begin{pmatrix} I_n \\ O_{m-n,n} \end{pmatrix} 
\end{equation}
is an element in the Stiefel manifold $\U(m)\slash \U(m-n)$. The orthonormality of the columns of $\widetilde{U}$ can be enforced by $n^2$ Dirac delta functions (recall that $\delta^{(2)}$ is the one for complex numbers). In this way we integrate in
\eqref{EfYYstar} over the larger space $\mathcal M(m,n)$. See also the discussion in \cite{CK}.

The complex matrix $M$ can be partitioned into two blocks
\begin{equation}
 M=\begin{pmatrix} M_1 \\ M_2 \end{pmatrix}
\end{equation}
with $M_1$ an $(n+\nu)\times n$ complex matrix and $M_2$ an $(m-n-\nu)\times n$ complex matrix. Then we have to calculate
\begin{align}
 \mathbb E \left[f(YY^*) \right]&=
 	\int_{\mathcal M(n+\nu,m)} f \left(M_1 XX^*M_1^*
			\right) \nonumber \\
			& \qquad \qquad \times\prod\limits_{1\leq j<k\leq n+\nu} \delta^{(2)}(\{M^*M\}_{jk})\prod_{j=1}^{n+\nu} \delta(\{M^*M\}_{jj}-1) dM.  \label{EfYYstar.b}
\end{align}

\subsection{Proof of Theorem \ref{propTruncation}}

To establish \eqref{toprove1} we proceed in six steps.

\paragraph{Step 1: Matrix delta function}

In the first step we rewrite the Dirac delta functions in \eqref{EfYYstar} as  Fourier-Laplace transforms \cite{CK}
where 
\begin{eqnarray}
&&\hspace*{-1.5cm}\prod\limits_{1\leq j<k\leq n} \delta^{(2)}(\{M^*M\}_{jk})\prod_{j=1}^{n} \delta(\{M^*M\}_{jj}-1) \nonumber \\
 &&= \lim_{t\to0}\frac{1}{2^{n}\pi^{n^2}}\int_{\mathcal H(n)} \exp[\Tr (I_{n}-iK)( I_{n}-M^*M )-t \Tr K^2] dK \label{Dir-rep}
\end{eqnarray} 
with an integration over the space $\mathcal H(n)$ of Hermitian $n \times n$ matrices $K$.
For an integration over the whole group $\U(m)$, i.e. $m=n$, this integration is equal to the one in \cite[formula (13)]{BM}. The shift of the matrix $K$ by $iI_n$ ensures the absolute integrability of the integral over $M$ and the Gaussian incorporating the limit in the auxiliary variable $t$ guarantees the absolute integrability in $K$. Note that the limit has to be understood as a limit in the weak topology meaning that we have to integrate over $M$, first, and, then, take the limit $t\to0$.
Hence the integral~\eqref{EfYYstar.b} reads
\begin{multline}
 \mathbb E \left[f(YY^*) \right] \propto \lim_{t\to0}\int_{\mathcal M(m,n)} f \left(M_1 XX^*M_1^*
			\right)\label{EfYYstar.c}\\
			\times\int_{\mathcal H(n)} \exp[\Tr (I_{n}-iK)(I_{n}-M_1^*M_1-M_2^*M_2 )-t \Tr K^2] dKdM
\end{multline}
up to a constant only depending on $m$, $n$ and $\nu$.

Both integrals are absolutely integrable. Therefore we can interchange the integrals. The integral over the matrix $M_2$  is a Gaussian integral yielding
\begin{multline}
 \mathbb E \left[f(YY^*) \right] \propto 
 \lim_{t\to0}\int_{\mathcal H(n)} \int_{\mathcal M(n+\nu,n)} f \left(M_1 XX^*M_1^*
			\right)\det(I_{n}-iK)^{-m+n+\nu} \\
	 \times\exp[\Tr (I_{n}-iK)(I_{n}-M_1^*M_1)-t\Tr K^2] dM_1dK. \label{EfYYstar.d}
\end{multline}
Finally we take $t\to0$. This can be done because $f$ is polynomial. In the case $f=1$ the integral over $M_1$ yields an additional factor $\det(I_{n+\nu}-iK)^{-n}$ ensuring the absolute integrability also at $t=0$. Since the function $f$ is polynomial we know 
that $f \left(M_1 XX^*M_1^* \right)$ is a polynomial in the matrix entries of $M_1$. Thus the Gaussian integral over $M_1$ yields a polynomial in $(I_{n+\nu}-iK)^{-1}$ with the lowest order to be $\det(I_{n+\nu}-iK)^{-n}$. Therefore the integrand of the $K$-integral after integrating over $M_1$, first, is indeed absolutely integrable also at $t=0$ such that
\begin{multline}
 \mathbb E \left[f(YY^*) \right] \propto \int_{\mathcal H(n)} \int_{\mathcal M(n+\nu,n)} f \left(M_1 XX^*M_1^* 	\right)\det(I_{n}-iK)^{-m+n+\nu} \\
		\times\exp[\Tr (I_{n}-iK)(I_{n}-M_1^*M_1)] dM_1dK. \label{EfYYstar.e}
\end{multline}
We underline that now the order of the integrals is crucial and cannot be interchanged.

\paragraph{Step 2: Changes of variable}

The change of variables $M_1 \mapsto S =  M_1X$ has the Jacobian
\begin{equation}
\det(X^*X)^{-n - \nu} = \prod_{j=1}^n x_j^{-n-\nu} 
\end{equation}
and the change of variables $K \mapsto \tilde{K} = X^{-1} K (X^{-1})^*$ on the
space of Hermitian matrices yields
\begin{equation}
\det(X^*X)^n = \prod_{j=1}^n x_j^n.
\end{equation}
Applying this to \eqref{EfYYstar.e} (and dropping the tilde from $\tilde{K}$), 
we obtain
\begin{align}
\mathbb E \left[f(YY^*) \right]&\propto \prod_{j=1}^n x_j^{-\nu} \int_{\mathcal H(n)} dK  \det(I_n-iX K X^*)^{-(m-n-\nu)} \nonumber\\
	& \qquad \times\int_{\mathcal M(n+\nu,n)} 	dS f(SS^*) 		e^{\Tr(I_n-iX K X^*)(I_n-(X^{-1})^* S^*S X^{-1})} \nonumber\\
	&= \prod_{j=1}^n x_j^{-m+n} \int_{\mathcal H(n)} dK \det( (X^*X)^{-1}-iK)^{-(m-n-\nu)} \nonumber\\
	& \qquad \times\int_{\mathcal M(n+\nu,n)} dS f(SS^*) 	e^{\Tr((X^*X)^{-1} -i K)  (X^*X-S^*S)} \nonumber\\
	&=  \prod_{j=1}^n x_j^{-m+n} \int_{\mathcal H(n)} \Psi(K + i (X^*X)^{-1} - i I_n) dK,   \label{intFU5}
\end{align}
where $\Psi$ is defined by
\begin{equation} \label{def:Psi} 
	\Psi(K) := \det(I_n-iK)^{-(m - n - \nu)} \int_{\mathcal M(n+\nu,n)}  f(S S^*) 
				e^{\Tr(I_n-iK)(X^*X-S^*S)} dS.
	\end{equation}

\paragraph{Step 3: Shift in the matrix $K$}

In this step we prove the following Lemma.

\begin{lemma} \label{lemma:shift}
For any matrix $A$ with $\Im A:=(A-A^*)/2i > -I_n$ (which means that  $I_n + \Im A$ is a positive definite Hermitian
matrix), we have
\begin{equation} \label{Psi-shift} 
	\int_{\mathcal H(n)} \Psi(K + A) dK = \int_{\mathcal H(n)} \Psi(K) dK. 
	\end{equation}
\end{lemma}

\begin{proof}
If $A$ is Hermitian then we can simply apply the linear change of variables $K +A \mapsto K$
to obtain \eqref{Psi-shift}. Therefore we may restrict to the case $A = i B$ with $B$ Hermitian and $B +I_n$ positive definite.
We may also restrict to the case where $B$ is a diagonal matrix. To see this we write $\Psi_X$ 
to indicate that the definition~\eqref{def:Psi} depends on $X$.  Then for a unitary matrix $U$ one has 
\begin{equation} \label{PhiUX} 
	\Psi_X(U K U^*) = \Psi_{U^*X}(K), 
	\end{equation}
which follows from inserting $UKU^*$ into  the definition \eqref{def:Psi} and 
changing variables $S \mapsto  SU^*$. Recall that $f(SS^*)=f(USS^*U^*)$ for all $U\in\U(n)$ because $f$ only depends on the squared singular values of $S$.
The invariance~\eqref{PhiUX}  implies by the unitary invariance
of $dK$, see e.g.\ \cite{Dei}, that 
\begin{align} 
	\int_{\mathcal H(n)} \Psi_X(K + i B) dK  = \int_{\mathcal H(n)} \Psi_{U^*X} (K+i U^*BU) dK.
	\end{align}
and we may choose the unitary matrix $U$ so that $U^*BU$ is diagonal.	

Let $p=m-n-\nu>0$ and $A=iB = i\diag(b_1, \ldots, b_n)$ a diagonal matrix with $b_j > -1$,
for $j=1, \ldots, n$. Note that for $0 \leq t \leq 1$,
\begin{align} \nonumber 
	\left| \Psi(K + it B) \right| & \leq \left|\det (I_n + t  B- iK)\right|^{-p}
	 \int_{\mathcal H(n)}  \left|f(S S^*)\right| e^{\Tr(I_n + tB)(XX^* -S^*S)} dS \\
	& \leq C_0 \left|\det (I_n + t  B- iK)\right|^{-p}.
	\label{Phiestimate0}
	\end{align}
The prefactor is a finite constant $C_0 > 0$ depending on $B$ and $X$ but is independent of $t \in [0,1]$.
Since all $b_j > -1$, it is clear that $I_n + tB$ is a positive definite matrix. Hence
there exists a constant $C_1 > 0$, independent of $t$, such that $\left| \det(I_n+tB) \right| > C_1$ resulting in
\begin{equation} \label{detestimate} 
	\left|\det (I_n + t B- iK) \right| > C_1 \left|\det(I_n - i M_t) \right|,
	\end{equation}
where
\begin{equation} \label{defMt} 
	M_t := (I_n+tB)^{-\frac{1}{2}} K (I_n+tB)^{-\frac{1}{2}} 
	\end{equation}
is a Hermitian matrix. 
Let $\lambda_1(t), \ldots, \lambda_n(t)$ denote the eigenvalues of $M_t$. 
Then it is easy to check that, since the eigenvalues $\lambda_j(t)$ are real,
\begin{align} 
		\left|\det(I_n - i M_t) \right| & = \prod_{k=1}^n \left|1-i \lambda_k(t) \right|=\prod_{k=1}^n \sqrt{1+\lambda_k^2(t)}\nonumber\\
		&\geq \max_{k = 1, \ldots, n} \left|\lambda_k(t) \right| = \max_{x \in \mathbb{C}^n \setminus \{0\}} 
				\frac{\left|\langle M_t x, x \rangle \right|}{\langle x, x\rangle} 
			\label{detestimate2}
\end{align}
by the properties of Rayleigh quotients. Taking $x = x_j = (I_n+ tB)^{\frac{1}{2}}e_j$ in \eqref{detestimate2}
and noting that $\langle M_t x_j, x_j\rangle = K_{jj}$, see \eqref{defMt}, we obtain
\begin{equation} \label{detestimate3}
	\left|\det(I_n - i M_t) \right| \geq 
		\max_{j = 1, \ldots, n} \frac{|K_{jj}|}{\langle x_j, x_j \rangle}. 
		\end{equation}
The vectors $x_j$ depend on $t$, but their norms are uniformly bounded from below $\langle x_j, x_j \rangle\geq 1+t b_{\min}>C>0$ for $t\in [0,1]$ where $b_{\min}=\min_{j=1,\ldots,n}b_j$ and $C=1$ if $b_{\min}\geq0$ or $C=1-b_{\min}$ if $0\geq b_{\min}>-1$.		
	Then, combining \eqref{Phiestimate0}, \eqref{detestimate} and \eqref{detestimate3}, we obtain
\begin{equation} \label{Phiestimate} 
	\left| \Phi(K+tiB) \right| \leq  \frac{C_2}{\max_{j=1, \ldots, n} |K_{jj}|^p},
	\end{equation}
for some constant $C_2 > 0$, independent of $t$.

In the integral $\int_{\mathcal H(n)} \Psi(K) dK$ we first do the integration over the
diagonal elements $K_{jj}$ for $j=1, \ldots, n$. The integrand is analytic in each of the $K_{jj}$.
Because of the estimate \eqref{Phiestimate} with $p \geq 1$, the integral can be deformed from the real line to 
the horizontal line in the complex $K_{jj}$-plane with imaginary part equal to $b_j$. We do this for all diagonal entries
resulting in \eqref{Psi-shift}.
\hfill{$\Box$}
\end{proof}

Lemma \ref{lemma:shift}  can be applied to \eqref{intFU5} because 
$A = i(X^*X)^{-1}-iI_n$ satisfies $\Im A = (X^*X)^{-1} - I_n > -I_n$. 
Hence we have
\begin{multline}
\mathbb E \left[f(YY^*) \right] \propto \prod_{j=1}^n x_j^{-m+n} \int_{\mathcal H(n)} dK \det( I_n-iK)^{-(m-n-\nu)}
  \\
	\times\int_{\mathcal M(n+\nu,n)} dS f(SS^*) 	e^{\Tr(I_n-i K)  (X^*X-S^*S)}. \label{intFU6}
\end{multline}

\paragraph{Step 4: Singular value decomposition of $S$}
We take the singular value decomposition $S = V_1 \Sigma V_2$ where $V_1\in\U(n+\nu)/[\U^n(1)\times\U(\nu)]$, $V_2\in\U(n)$,
and $\Sigma$ is a diagonal matrix with the singular values of $S$.
The Jacobian of this transformation is proportional to $\Delta(y)^2 \prod_{j=1}^n y_j^{\nu}$,
see \cite{EdRa}, where $y_1, \ldots, y_n$ are the squared singular values. Recall that $f$ is defined as a symmetric
function on the eigenvalues.
Then \eqref{def:Psi} reads
\begin{multline} \label{Psi-svd}
\Psi(K) 	\propto \frac{e^{\Tr(I_n-iK)(X^*X)}}{\det(I_n-iK)^{m-n-\nu}} \int_{[0,\infty)^n} f(y_1, \ldots, y_n) \prod_{j=1}^n y_j^{\nu}
		\Delta(y)^2 dy_1 \ldots dy_n \\
		\times \int_{\mathcal U(n)} e^{-\Tr(I_n-iK) V_2^* \Sigma^2 V_2} dV_2
		\end{multline}
The integral over $V_2 \in \mathcal U(n)$ is a HCIZ integral \eqref{HCIZ-formula}. 
Let $\lambda_1, \ldots, \lambda_n$ be the eigenvalues of $K$, then the HCIZ integral yields a term
proportional to $\det [e^{-(1-i \lambda_j)y_k} ]$ $/[\Delta(\lambda) \Delta(y)]$.
We end up with
\begin{multline} \label{Psi-svd2}
\Psi(K) \propto \frac{e^{\Tr(I_n-iK)(X^*X)}}{\det(I_n-iK)^{m-n-\nu} \Delta(\lambda)} \\
	\times \int_{[0,\infty)^n} f(y_1, \ldots, y_n) \prod_{j=1}^n y_j^{\nu}
		\Delta(y) \det\left[e^{-(1-i \lambda_j)y_k} \right]_{j,k=1}^n dy_1 \ldots dy_n.
		\end{multline}
		
\paragraph{Step 5: Eigenvalue decomposition of $K$}

Next, we decompose $K = V_0 \Lambda V_0^*$ in a unitary matrix  $V_0\in\U(n)/\U^{n}(1)$ 
and its eigenvalues $\Lambda = \diag(\lambda_1, \ldots, \lambda_n)$. The Jacobian is proportional to $\Delta(\lambda)^2$,
see e.g. \cite{Dei},
and by \eqref{intFU6} and \eqref{Psi-svd2} we have
\begin{multline} \label{intFU8}
	\mathbb E \left[f(YY^*) \right]\propto 
	\prod_{j=1}^n x_j^{-m+n}
	\int_{\mathbb R^n} \prod_{j=1}^n (1- i \lambda_j)^{-m+n+\nu}  \Delta(\lambda) d \lambda_1 \cdots d \lambda_n
	  \\	\times
	\int_{[0,\infty)^n} f(y_1, \ldots, y_n)  \prod_{j=1}^n y_j^{\nu} \Delta(y) 
		\det\left[e^{-(1-i \lambda_j)y_k} \right]_{j,k=1}^n  d y_1 \cdots dy_n \\
 \times	\int_{\mathcal U(n)} e^{\Tr(I_n-iV_0 \Lambda V_0^*)(X^*X)}dV_0.
	\end{multline}
The $V_0$ integral is again a HCIZ integral \eqref{HCIZ-formula},
and it gives a contribution proportional to $\det[ e^{(1-i \lambda_j) x_k} ]/[\Delta(\lambda) \Delta(x)]$.
Plugging this term into \eqref{intFU8} and noting that we may change
the order of integration at this stage, we find
\begin{align} \label{intFU9} 
	\mathbb E \left[f(YY^*) \right]= \int_{[0,\infty)^n} f(y_1, \ldots, y_n) p_{X,Y}(x,y) dy_1 \cdots dy_n
	\end{align}
where
\begin{multline} \label{Py1yn} 
	p_{X,Y}(x,y)\propto	\left(\prod_{j=1}^n x_j^{-m+n} \right)  \left( \prod_{j=1}^n y_j^{\nu} \right) 
		\frac{\Delta(y)}{\Delta(x)} \times \\
	\int_{\mathbb R^n}  \prod_{j=1}^n (1- i \lambda_j)^{-(m-n-\nu)} 
	\det\left[ e^{(1-i \lambda_j) x_k}\right]_{j,k=1}^n  \det\left[ e^{-(1-i \lambda_j) y_k} \right]_{j,k=1}^n  d \lambda_1 \cdots d \lambda_n.
	\end{multline}

\paragraph{Step 6: Andr\'eief formula}
We calculate the integral over $\lambda_1, \ldots, \lambda_n$ in \eqref{Py1yn} via the Andr\'eief formula \eqref{Andreief}.
This gives the determinant
\begin{align} \label{detintegrals}
	n! 
	 \det \left[ \int_{-\infty}^{\infty} 
		\frac{e^{(x_k-y_j) (1-i\lambda)}}{(1-i \lambda)^{m-n-\nu}} d \lambda \right]_{j,k=1}^n
		\end{align}
with integrals of the form~\eqref{Ingham-Siegel-nis1}.
The result is
\begin{equation}
\int_{-\infty}^{\infty} \frac{e^{(x_k-y_j) (1-i\lambda)}}{(1-i \lambda)^{m-n-\nu}} d \lambda
		=  \frac{2\pi}{(m-n-\nu-1)!} \, (x_k - y_j)_+^{m-n-\nu-1}.
\end{equation}
We use this in \eqref{intFU9} and \eqref{Py1yn} and obtain  \eqref{toprove1}. This concludes concludes the proof of the Theorem \ref{propTruncation}.
 \hfill{$\Box$}

\section{Proofs of Proposition \ref{PkQkintegrals} and Theorem \ref{Knlimit}} \label{sec:proof3}

We first prove Proposition \ref{PkQkintegrals} and Theorem \ref{Knlimit} afterwards.

\begin{proof}[Proof of Proposition \ref{PkQkintegrals}] \quad
Recall that $P_k$ and $Q_k$ are defined by \eqref{defPk} and \eqref{defQk}, respectively, for $k=0, \ldots, n-1$. 
Then $Q_k$ is in the linear span of $w_1^{(r)}, \ldots, w_n^{(r)}$ since
it is of the form \eqref{wkr-span} with $q(s) = (s-k)_k$ a polynomial of degree $k \leq n-1$.
Also note that the integrand in \eqref{defPk} has poles at $t=0, \ldots, k$, and no other poles inside $\Sigma_k$.
Thus if we evaluate \eqref{defPk} by means of the residue theorem and find contributions 
from $t=0, \ldots, k$, only. Hence the function $P_k$ is a polynomial of degree $k$.
It remains to verify the biorthogonality \eqref{PjQk-biorthogonal}.

Let $l,k = 0, 1, \ldots, n-1$. 
From \eqref{defPk} we have
\begin{multline} \int_0^1 P_l(x) Q_k(x) dx  \\
	 = \frac{1}{2\pi i} \oint_{\Sigma_l} \frac{1}{(t-l)_{l+1}} 
		\prod_{j=1}^r \frac{\Gamma(t+1+m_j-n)}{\Gamma(t+1+\nu_j)} 
		\int_0^1 x^t Q_k(x) dx dt. \label{PjQk-biorthogonal1}
		\end{multline}
The moments of $Q_k$ are given by the general identity for the moments of the Meijer G-function,
\begin{equation} 
\int_0^{1} x^{s-1} Q_k(x) dx =  (s-k)_k \prod_{j=1}^r \frac{\Gamma(s+\nu_j)}{\Gamma(s+m_j-n)}.
	\end{equation}
This identity can be plugged into \eqref{PjQk-biorthogonal1} which cancels a lot of $\Gamma$-factors,
\begin{equation} \label{PjQk-biorthogonal2}
 \int_0^1 P_l(x) Q_k(x) dx  
	=  \frac{1}{2\pi i} \oint_{\Sigma_l} \frac{(t+1-k)_k}{(t-l)_{l+1}} 	dt=  \frac{1}{2\pi i} \oint_{\Sigma_l} \frac{\Gamma[t-l]}{\Gamma[t+1-k]} 	dt.
	\end{equation}
For $l < k$, the integrand in \eqref{PjQk-biorthogonal2} is a polynomial	
and the integral is zero by Cauchy's theorem.
For $l \geq k$, there are poles at $t=k, \ldots, l$ which are all inside the contour $\Sigma_l$.
The integrand is a rational function that behaves like $O(t^{k-l-1})$ as $|t| \to \infty$. Therefore
by simple residue calculation at infinity, the integral also vanishes if $l > k$ and it is equal to $1$ if $l = k$. 
Thus \eqref{PjQk-biorthogonal} is satisfied.

Inserting \eqref{defPk} and \eqref{defQk} into \eqref{Knsum}, and noting that we can take
the same contour $\Sigma_n$ for every $k$ in \eqref{defPk}, we obtain  a double integral for $K_n$,
\begin{multline} 
	K_n(x,y) 
	= 
\frac{1}{(2\pi i)^2}
		\int_{C} ds \oint_{\Sigma_n} dt \prod_{j=1}^r \frac{\Gamma(s+\nu_j) \Gamma(t+1+m_j-n)}{\Gamma(s+m_j-n) \Gamma(t+1+\nu_j)}   
			\sum_{k=0}^{n-1} \frac{(s-k)_k}{(t-k)_{k+1}}
		x^t 		y^{-s} \\
=\frac{1}{(2\pi i)^2}
		\int_{C} ds \oint_{\Sigma_n} dt \prod_{j=1}^r \frac{\Gamma(s+1+\nu_j) \Gamma(t+1+m_j-n)}{\Gamma(s+1+m_j-n) \Gamma(t+1+\nu_j)}   
			\sum_{k=0}^{n-1} \frac{(s+1-k)_k}{(t-k)_{k+1}}
		x^t 		y^{-s-1}, \label{a}
		\end{multline}
where we made the change of variables $s \mapsto s+1$.
The summation can be simplified, because of the telescoping sum
\begin{equation}\label{b}
\sum_{k=0}^{n-1} \frac{(s+1-k)_k}{(t-k)_{k+1}}
	= \frac{1}{s-t} \left(\frac{\Gamma(s+1)}{\Gamma(t+1)} \frac{\Gamma(t+1-n)}{\Gamma(s+1-n)} - 1 \right),
\end{equation}
which can be readily checked by complete induction.
The contours $\Sigma_n$ and $C$ do not intersect. Therefore the term $t=s$ is not a pole inside the contour $\Sigma_n$. The double integral~\eqref{a} splits into two terms due to \eqref{b}. For the second term the contour integral over $t$ does not encircle any pole because of $m_j-n>\nu_j$ and $m_j-n,\nu_j\in\mathbb{N}_0$. Thus this term vanishes. The remaining term is the one shown in \eqref{Knintegral}.

The expression \eqref{Knintegral} in terms of Meijer G-functions is obtained by noticing that
\begin{equation} \label{stInt}
\frac{1}{s-t} = -\int_0^1 u^{t-s-1} du.
\end{equation}
Interchanging the $s$ and $t$ integral with the $u$ integral and using the  definition of the Meijer G-functions~\eqref{MeijerG-def} the identity follows, see also the proof of Theorem 5.3 in \cite{KuZh}. This concludes the proof.
\hfill{$\Box$}
\end{proof}

\begin{proof}[Proof of Theorem \ref{Knlimit}] 

We employ the following asymptotic behavior of a ratio of Gamma functions
\begin{equation} \label{Gammaratiolimit} 
\frac{\Gamma(t+1 + m_j-n)}{\Gamma(s+1+m_j-n)} = 
	\begin{cases} \displaystyle \frac{\sin \pi s}{\sin \pi t} \, n^{t-s} \left(1 + O\left(\frac{1}{n}\right)\right), 
		& \text{ for } j = 0, \\ \displaystyle
		(m_j-n)^{t-s} \left(1 + O\left(\frac{1}{m_j-n}\right)\right), & 
		\text{ for } j \in \{1, \ldots, r \} \setminus J
		\end{cases} \end{equation}
as $n \to \infty$. This follows as in the proof of Theorem 5.3 in \cite{KuZh},
since $m_0 = 0$ and $m_j - n \to \infty$ as $n \to \infty$ for $j \in \{1, \ldots, r \} \setminus J$.

In the double  integral formula in \eqref{Knintegral}  we deform the contour $\Sigma_n$ to $\Sigma$, and obtain
\begin{multline} \label{Knscaled}
\frac{1}{c_n} K_n \left( \frac{x}{c_n}, \frac{y}{c_n} \right) \\
	= \frac{c_n^{s-t}}{(2\pi i)^2}
		\int_{C} ds \oint_{\Sigma} dt \prod_{j=0}^{r} 
		\frac{\Gamma(s+1+\nu_j)}{\Gamma(t+1+\nu_j)}
			\prod_{j=0}^{r}  \frac{\Gamma(t+1+m_j-n)}{\Gamma(s+1+m_j-n)} 
		\frac{x^t y^{-s-1}}{s-t}.
		\end{multline}
Because of definition \eqref{defcn} we have
\begin{multline} 
	c_n^{s-t} \prod_{j=0}^{r} \frac{\Gamma(t+1+m_j-n)}{\Gamma(s+1+m_j-n)} 
		= \left( n^{s-t} \frac{\Gamma(t+1+m_0-n)}{\Gamma(s+1+m_0-n)}\right) \\
	\times			\prod_{j \not\in J} \left( (m_j-n)^{s-t} \frac{\Gamma(t+1+m_j-n)}{\Gamma(s+1+m_j-n)} \right)
				\prod_{j \in J} 
				\frac{\Gamma(t+1+m_j-n)}{\Gamma(s+1+m_j-n)}.
\end{multline}
Each of the factors in the product has a finite limit as $n \to \infty$, cf. \eqref{Gammaratiolimit},
and the full product tends to  
\begin{equation}
 \frac{\sin \pi s}{\sin \pi t} \ \prod_{k=1}^q \frac{\Gamma(t+1+\mu_k)}{\Gamma(s+1+\mu_k)}
 \end{equation}
since $m_j-n = \mu_k$ if $j = j_k \in J$.
It is then allowed to take the limit inside the double integral in \eqref{Knscaled}
and deform the contour $C$ to the vertical line $\Re s = -1/2$ by analyticity
and the decay of the integrand at infinity in the $s$-variable. This leads to \eqref{scalinglimit}.

The expression in terms of Meijer G-functions is obtained by using the reflection formula of the Gamma function, $\sin\pi z =\pi/[\Gamma(1-z)\Gamma(z)]$ together with \eqref{stInt}. The identity now follows along the same lines as in the proof of Proposition \ref{PkQkintegrals}.
\hfill{$\Box$}
\end{proof}

\section{Conclusions and Outlook}\label{conclusio}

We analyzed the singular value statistics of a matrix $X$ (fixed as well as randomly chosen)  multiplied by a truncated unitary matrix $T$ which is distributed by the induced Haar measure. Though we only considered a multiplication from the left side $TX$ one can easily generalize the results to a multiplication from both sides $T_{\rm L}XT_{\rm R}$ where $T_{\rm L}$ and $T_{\rm R}$ are truncations of two independent unitary matrices. The reason for such a simple generalization is the determinantal point process fulfilled by the joint probability density of the singular values. We proved that the  joint probability density of the squared singular values of $TX$ satisfies a polynomial ensemble if the joint probability density of the squared singular values of $X$ does. In particular with the help of our results one can calculate the squared singular value statistics of any product $T_{\rm L,1}\cdots T_{\rm L,r_{\rm L}}XT_{\rm R,1}\cdots T_{\rm R,r_{\rm R}}$ with $r_{\rm L},r_{\rm R}\in\mathbb{N}_0$ and $T_{\rm L,j}$ and  $T_{\rm R,j}$ truncations of independent unitary matrices and $X$ either fixed or another random matrix. Indeed one can also mix products of truncated unitary matrices with Ginibre matrices. The combination of Theorem~\ref{Ginibreproduct} and Corollary~\ref{TruncationTransformation}  yields a polynomial ensemble for the squared singular values of such a mixed product. We expect that also for such a mixed product the statistics are governed by Meijer G-functions since this particular kind of functions is closed under Mellin convolution as it was shown here for a pure product of truncated unitary matrices and studied in \cite{AkIpKi,AkKiWe} for a product of Ginibre matrices.

Our study shows that the determinantal point process also applies to a product with truncated unitary matrices. In particular one needs a group integral replacing the Harish~Chandra/ Itzykson-Zuber integral~\cite{Harish,ItZu} for which it is not immediate that the result can be written in terms of determinants, cf. Theorem~\ref{GRunitaryintegral}. This particular result is even more astounding when noticing that one does not effectively integrate over the whole unitary group but only over a subset. The reason for this is a Heaviside step function in the integrand. Harish Chandra made contact between group integrals and representation theory \cite{Harish}. It would be interesting if also something like this exists for the integral considered by us and can be explained by group theory.

Additionally we looked at the spectral statistics of a product of truncated unitary matrices in detail. For this product we calculated the kernel of the corresponding polynomial ensemble of the squared singular values at finite matrix size and in the hard edge limit at the origin. In a forthcoming publication we also derive the bulk, the soft and the hard edge statistics as it was very recently done for the Ginibre ensemble in \cite{LiWaZh}. The latter two statistics may appear at the lower and the upper bound of the support because the squared singular values of a product of truncated unitary matrices always live on the interval $[0,1]$. If the support touches either the origin or the upper bound $1$ one would expect hard edge statistics at these points.

Another generalization of our results refers to the restriction that the first matrix in the product of truncated unitary matrices has to satisfy $m_1>2n+\nu_1$, see Corollary~\ref{TruncationTransformation}. This matrix has not generally to be  the first matrix $T_1$. With the help of the discussion above it can be any matrix in the product. Nevertheless we have to assume that at least one truncated unitary matrix multiplied has to satisfy the condition to prove Corollary~\ref{TruncationTransformation} in the way we have done. An interesting question would be: What happens if this restriction is not met? From the spectral statistics of one truncated unitary matrix, see e.g. \cite{ZySo,For-book}, we know that some singular values are exactly located at $1$. Numerical simulations performed by us hint that this seems to be true also for a product of truncated unitary matrices. In the case that this is indeed true, the question arises about the algebraic structure. Does the determinantal point process carry over to a product $T_r\cdots T_1$ where the restriction $m_j>2n_j+\nu_j$ is not met for all $j=1,\ldots, r$?

\section*{Acknowledgements}

MK acknowledges partial financial support by the Alexander von Humboldt foundation. 
AK and DS are supported by KU Leuven Research Grant OT/12/073 and the
Belgian Interuniversity Attraction Pole P07/18. AK is also supported
by FWO Flanders projects G.0641.11 and G.0934.13.




\end{document}